\documentclass[11pt]{35}
\usepackage{amsthm, amsmath, amssymb, amsbsy, amsfonts, latexsym, euscript}
\usepackage{calrsfs} 
\usepackage{graphicx}

\DeclareMathAlphabet{\varmathbb}{U}{pxsyb}{m}{n}

\def\leq{\leqslant}

\def\geq{\geqslant}
\def\phi{\varphi}

\def\bar{\overline}

\def\kappa{\varkappa}

\newcommand{\D}{\mathrm{d}\kern0.2pt}%
\newcommand{\E}{\mathrm{e}\kern0.2pt} 
\newcommand{\ii}{\kern0.05em\mathrm{i}\kern0.05em}

\newtheorem{theorem}{\bf \indent Theorem}[section]
\newtheorem{proposition}{\bf \indent Proposition}[section]

\newtheorem{corollary}{\bf \indent Corollary}[section]

\theoremstyle{remark} 
\newtheorem{remark}{\bf \indent Remark}[section]

\numberwithin{equation}{section}

\begin{document}

\noindent {\Large \bf On relations between harmonic functions and solutions \\[3pt]
of the modified Helmholtz equation}

\vskip2mm

{\bf Nikolay Kuznetsov}

\vskip-2pt {\small Laboratory for Mathematical Modelling of Wave Phenomena}
\vskip-4pt {\small Institute for Problems in Mechanical Engineering} \vskip-4pt
{\small Russian Academy of Sciences} \vskip-4pt {\small V.O., Bol'shoy pr. 61, St
Petersburg, 199178} \vskip-4pt {\small Russian Federation} \vskip-4pt {\small
nikolay.g.kuznetsov@gmail.com}

\vspace{2mm}

\begin{quote}
\noindent We consider the $m$-dimensional modified Helmholtz equation and establish
two relations between its solutions in a bounded domain and harmonic functions. Both
relations essentially rely on properties of the Newtonian potential. Some other
characteristics of these solutions are also obtained.
\end{quote}

\vspace{-4mm}

{\centering \section{Introduction} }

\noindent In the present note, we consider real-valued $C^2$-solutions of the $m$-dimensional
modified Helmholtz equation:
\begin{equation}
\nabla^2 u - \mu^2 u = 0 , \quad \mu \in {\bf R} \setminus \{0\} ,
\label{Hh}
\end{equation}
$\nabla = (\partial_1, \dots , \partial_m)$ is the gradient operator, $\partial_i =
\partial / \partial x_i$. Unfortunately, it is not com\-monly known that these 
solutions are called panharmonic functions (or $\mu$-pan\-har\-monic functions) by
analogy with solutions of the Laplace equation. This convenient abbreviation coined
by Duffin \cite{D} will be used in what follows.
 
Undeservedly, panharmonic functions received much less attention than harmonic and
subharmonic. Recently, the author published several papers aimed at filling in this
gap; see \cite{Ku1}, \cite{Ku2}, \cite{Ku3}, \cite{Ku4} and \cite{Ku5}. In
particular, the following results about panharmonic functions are obtained in
\cite{Ku1} (some also formulated without proof in \cite{Ku4}): the mean value
formulae for spheres and balls (see Theorem~1.1 and Corollary~1.2, respectively,
formulated below), several corollaries of these formulae include the Liouville
theorem and the strong converse of Theorem 1.1. In the brief note \cite{Ku2},
asymptotic mean value properties of panharmonic functions are considered, whereas
the notes \cite{Ku3} and \cite{Ku5} deal with characterizations of balls via these
functions.

The aim of this paper is to describe relations between harmonic and panharmonic
functions and the motivation for this is the fundamental theorem on subharmonic
functions (published by F.~Riesz \cite{R} in 1930), which establishes the
decomposition of such a function into the sum of a harmonic function and a Newtonian
potential. (The result was proved by Riesz for functions of two variables, whereas
the general case can be found in \cite{HK}, Section~3.5.) It occurs that panharmonic
functions have many common properties with subharmonic ones and this is the reason
to apply methods developed for subharmonic functions in studies of panharmonic
ones.

To demonstrate the similarity between panharmonic and subharmonic functions we use
the mean value property. The following analogue of the Gauss theorem on the arithmetic
mean of a harmonic function over an $m-1$-dimensional sphere in ${\bf R}^m$ (the
original memoir \cite{G}, Article~20, deals with the case $m=3$) was recently
obtained for panharmonic functions; see \cite{Ku1}.

\begin{theorem}
Let $u$ be panharmonic in a domain $D \subset {\bf R}^m$, $m \geq 2$. Then for
every $x \in D$ the identity
\begin{equation}
M^\circ (x, r, u) =  a (\mu r) \, u (x) , \quad a^\circ (\mu r) = \Gamma \left(
\frac{m}{2} \right) \frac{I_{(m-2)/2} (\mu r)}{(\mu r / 2)^{(m-2)/2}} \, ,
\label{MM}
\end{equation}
holds for each admissible sphere $S_r (x);$ $I_\nu$ denotes the modified Bessel
function of order $\nu$.
\end{theorem}

Here and below the following notation and terminology are used. The open ball of
radius $r$ centred at $x$ is denoted by $B_r (x) = \{ y : |y-x| < r \}$; the latter
is called admissible with respect to a domain $D$ provided $\overline{B_r (x)}
\subset D$, and $S_r (x) = \partial B_r (x)$ is the corresponding admissible sphere.
If $u \in C^0 (D)$, then its spherical mean value over $S_r (x) \subset D$ is
\begin{equation}
M^\circ (x, r, u) = \frac{1}{|S_r|} \int_{S_r (x)} u (y) \, \D S_y =
\frac{1}{\omega_m} \int_{S_1 (0)} u (x +r y) \, \D S_y \, , \label{sm}
\end{equation}
where $|S_r| = \omega_m r^{m-1}$ and $\omega_m = 2 \, \pi^{m/2} / \Gamma (m/2)$ is
the total area of the unit sphere (as usual $\Gamma$ stands for the Gamma function),
and $\D S$ is the surface area measure. For $m=3$ identity \eqref{MM} (derived by
C.~Neumann \cite{NC} as early as 1896) is particularly simple because $a^\circ (\mu
r) = \sinh \mu r / (\mu r)$. Duffin independently rediscovered the proof (see
\cite{D}, pp.~111-112), but in two dimensions with $a^\circ (\mu r) = I_0 (\mu r)$.

It is easy to demonstrate that the function $a^\circ$ increases monotonically on
$[0, \infty)$ from $a^\circ (0) = 1$ to infinity. Indeed, the relation $[z^{-\nu}
I_\nu (z)]' = z^{-\nu} I_{\nu+1} (z)$ (see \cite{Wa}, p. 79), where the right-hand
side is positive for $z > 0$ and vanishes at $z = 0$, implies the monotonicity,
whereas the behavior at infinity is a consequence of the asymptotic formula
\begin{equation*}
I_\nu (z) = \frac{\E^z}{\sqrt{2 \pi z}} \left[ 1 + O (|z|^{-1}) \right] \, , \ \
|\arg z| < \pi /2 ,
\end{equation*}
valid as $|z| \to \infty$ (see \cite{Wa}, p. 80). Therefore, $a^\circ (\mu r) > 1$
for every admissible sphere $S_r (x)$, and so \eqref{MM} yields that $u (x) \leq
M^\circ (x, r, u)$ provided the panharmonic $u$ is nonnegative in~$D$. This implies
the following.

\begin{corollary}
Let a panharmonic function $u$ be nonnegative (nonpositive) in a domain $D$. Then
$u$ is subharmonic (superharmonic) in $D$.
\end{corollary}

The result immediately follows from the last inequality and the definition of
subharmonic (superharmonic) functions used by Gilbarg and Trudinger (see \cite{GT},
p.~23), which is sufficient in our context. Namely, a function $u \in C^0 (D)$ is
called subharmonic (superharmonic) in $D$ if for every admissible ball $B$ and every
function $h$ harmonic in $B$ and satisfying $u \leq h$ ($u \geq h$) on $\partial B$,
the same inequality holds throughout $\bar B$.

The converse of Corollary 1.1 is not true, because any nonzero constant is
subharmonic and superharmonic, but not panharmonic. Furthermore, the validity of the
Riesz decomposition theorem and some other results for nonnegative panharmonic
functions is a consequence of this corollary; see the precise formulations in
Section~2.
 
Another consequence of Theorem 1.1 is the following assertion (see \cite{Ku1},
Corollary~2.1).

\begin{corollary}
Let $D$ be a domain in ${\bf R}^m$, $m \geq 2$. If $u$ is panharmonic in $D$, then
\begin{equation}
M^\bullet (x, r, u) =  a^\bullet (\mu r) \, u (x) , \quad a^\bullet (\mu r) =
\Gamma \left( \frac{m}{2} + 1 \right) \frac{I_{m/2} (\mu r)}{(\mu r / 2)^{m/2}} \, ,
\label{MM'}
\end{equation}
for every admissible ball $B_r (x)$.
\end{corollary}

Here $M^\bullet (x, r, f) = |B_r|^{-1} \int_{B_r (x)} f (y) \, \D y$ is the mean
over the ball $B_r (x)$, whose volume is $|B_r| = m^{-1} \omega_m r^m$. Like
$a^\circ$, the function $a^\bullet$ increases monotonically on $[0, \infty)$ from
$a^\bullet (0) = 1$ to infinity.

Combining \eqref{MM} and \eqref{MM'}, one arrives at the identity
\begin{equation} 
a^\circ (\mu r) M^\bullet (x, r, u) = a^\bullet (\mu r) M^\circ (x, r, u) \, ,
\label{MMtil}
\end{equation}
which couples the mean values over spheres and balls for a panharmonic $u$. In~this
identity, the ratio of coefficients $a^\circ (\mu r) / a^\bullet (\mu r)$ tends to
unity in the limit as $\mu \to 0$, thus reducing \eqref{MMtil} to the well-known
formula equating the mean values of a har\-monic function over spheres and balls. On
the other hand, it is long known that if the latter mean values are equal for an
arbitrary continuous function and all admissible balls, then the function is
harmonic; see \cite{BR}, where this assertion was proven in two dimensions, and
\cite{Ku}, Theorem~1.8, for the general case.

It occurs that, under suitable assumptions, identity \eqref{MMtil} implies that $u$
is panharmonic (see Theorem 3.2 below; it was announced without proof in \cite{Ku4})
like the mentioned assertion for harmonic functions.

{\centering \section{Properties of positive panharmonic functions} }

It is worth to recall the Riesz decomposition theorem for subharmonic functions
(see, for example, \cite{HK}, Theorem~3.9), which involves the fundamental solution
$E_m (x - y)$ of the Laplace equation equal to $[ (2 - m) \, \omega_m |x -
y|^{(m-2)} ]^{-1}$ when $m \geq 3$ and to $(2 \pi)^{-1} \log |x - y|$ when $m=2$.

\begin{theorem}
If $u$ is subharmonic in a domain $D \subset {\bf R}^m$, $m \geq 2$, then there
exists a unique Borel measure $\mathbf m$ in $D$ such that for any compact set $K
\subset D$
\begin{equation}
u (x) = \int_K E_m (x - y) \, \D \mathbf m (y) + h (x) , \quad x \in \mathrm{int} K
, \label{Ri}
\end{equation}
where $\mathrm{int} K$ is the interior of $K$ and $h$ is harmonic there.
\end{theorem}

\begin{remark}
It follows from Treves's considerations (see \cite{T}, pp.~288--289) that if $u \geq
0$ is subharmonic in a bounded domain $D$, then formula \eqref{Ri} holds with $K$
changed to $D$, whereas $\D \mathbf m (y) = \nabla^2 u (y) \, \D y$ and $h$ is the
positive least harmonic majorant of $u$ in~$D$; for its definition see also
\cite{AG}, p.~79.
\end{remark}

Now we are in a position to formulate the following.

\begin{theorem}
Let $u \geq 0$ be $\mu$-panharmonic in a domain $D \subset {\bf R}^m$, $m \geq 2$,
then \eqref{Ri} takes the following form:
\begin{equation}
h (x) = u (x) - \mu^2 \int_K E_m (x - y) \, u (y) \, \D y , \quad x \in \mathrm{int}
K . \label{Rie}
\end{equation}
Here $K \subset D$ is a compact set and $h$ is harmonic in $\mathrm{int} K$.

If $D$ is bounded and, besides, $u \in C^0 (\bar D)$, then \eqref{Rie} is valid in
the whole $D$ with the integral over $D$, whereas $h \geq 0$ is the least harmonic
majorant of $u$ in $D$.
\end{theorem}

\begin{proof}
According to Corollary 1.1, $u$ is subharmonic in $D$, and so it has the Riesz
decomposition \eqref{Ri}. Applying the Laplacian to both sides of \eqref{Ri} and
taking into account equation \eqref{Hh} on the left-hand side and using the
harmonicity of $h$ and the definition of $E_m$ on the right, we see that $\mathbf m$
is proportional to the Lebesgue measure with the coefficient $\mu^2 u$ (cf.
Remark~2.1). Now \eqref{Rie} follows by rearranging.

The second assertion is obvious, whereas the last one is a consequence of the
considerations mentioned in Remark~2.1.
\end{proof}

Our next result involves mean values over a domain as well as over its boundary. In
this case, one can hardly expect an identity analogous to \eqref{MMtil} to be valid
for panharmonic functions in a domain distinct from a ball. Indeed, Bennett \cite{B}
proved the following.

\begin{theorem}
Let $D \subset {\bf R}^m$ be a bounded domain with sufficiently smooth boundary.~If
\[ |D|^{-1} \int_D h (y) \, \D y = |\partial D|^{-1} \int_{\partial D} h (y) \, \D S_y
\]
for every $h \in C^2 (\bar D)$ harmonic in $D$, then $D$ is an open ball.
\end{theorem}

A similar conjecture based on identity \eqref{MMtil} for panharmonic functions is
made in \cite{Ku5}. At the same time, an inequality holds between the mean values of
nonnegative panharmonic functions in a bounded domain under a suitable assumption
about its boundary.

\begin{proposition}
Let $D \subset {\bf R}^m$ be a bounded domain satisfying the exterior sphere
condition uniformly on $\partial D$. Then there exists a constant $c \in [1,
\infty)$, depending on $D$ and~$\mu$, and such that
\begin{equation}
|D|^{-1} \int_D u (y) \, \D y \leq c \, |\partial D|^{-1} \int_{\partial D} u (y) \,
\D S_y \label{FM}
\end{equation}
for every nonnegative panharmonic function $u \in C^0 (\bar D)$.
\end{proposition}

In view of Corollary 1.1, inequality \eqref{FM}, like Theorem~2.2, is a consequence
of the corresponding theorem proved for subharmonic functions; see \cite{FM},
p.~195. Moreover, if $D$ is a ball $B_r$ (there is no need to specify its center),
then equality takes place in \eqref{FM} with
\begin{equation}
c = \frac{a^\bullet (\mu r)}{a^\circ (\mu r)} = \frac{m I_{m/2} (\mu r)}{\mu r
I_{(m-2)/2} (\mu r)} < \frac{m}{\mu r} \, . \label{18}
\end{equation}
Here the equalities follow from identity \eqref{MMtil} and formulae \eqref{MM} and
\eqref{MM'}, whereas the inequality is a consequence of the definition of $I_\nu$.
This not only demonstrates that $c$ depends on $\mu$, but also improves
Proposition~2.1 for balls provided $\mu r > m$. It occurs that $c$ can be
arbitrarily small when either $r$ ($\mu$ fixed) or $\mu$ ($r$ fixed) is sufficiently
large (or both are sufficiently large).

{\centering \section{Characterizations of panharmonic functions} }

Let $D$ be a bounded domain in ${\bf R}^m$, $m \geq 3$; for $u \in L^2 (D)$ we
define the operator:
\[ (T u) (x) = \int_D E_m (x - y) \, u (y) \, \D y , \quad x \in D .
\]
Its symmetric kernel is positive after dropping the negative coefficient, which
yields well-known properties; see, for example, \cite{M}, Chapter~7.

Thus, the second assertion of Theorem 2.2 admits an interpretation in terms of the
domain and range of $I - \mu^2 T$ within the Banach space $C^0 (\bar D)$, in which
$T$ is compact (as usual, $I$ stands for the identity operator). Namely, $I - \mu^2
T$ maps the cone of nonnegative $\mu$-panharmonic functions into the cone of
nonnegative harmonic functions.

Let us consider whether there exists an inverse mapping: harmonic $\mapsto$
$\mu$-panharmonic functions. To this end we introduce the integral equation
\begin{equation}
u (x) - \lambda (T u) (x) = h (x) , \quad x \in D , \ \ \lambda \in {\bf R} , 
\label{21}
\end{equation}
where $u, h \in L^2 (D)$. This is a natural setting because the operator $T$ has a
weakly singular kernel, and so is compact and self-adjoint, whereas $-T$ is a
positive operator in this space. We recall that these properties of $T$ imply that
it has a sequence $\{\lambda_n\}_1^\infty$ of characteristic values each having a
finite multiplicity; moreover, these values are real negative numbers such that
$|\lambda_n| \to \infty$ as $n \to \infty$. Finally, if $\lambda \neq \lambda_n$ for
$n = 1,2,\dots$ (in particular, if $\lambda > 0$), then for any $h \in L^2 (D)$
equation \eqref{21} has a unique solution $u \in L^2 (D)$, which can be represented
by virtue of the resolvent kernel. Taking into account these facts, we formulate and
prove the following assertion, in which $C^{0, \alpha} (\bar D)$ stands for the
Banach space of functions that are H\"older continuous with exponent $\alpha \in (0,
1)$.

\begin{theorem}
Let $D$ be a bounded Lipschitz domain in ${\bf R}^m$, $m \geq 3$, and let $\lambda =
\mu^2 > 0$ in equation \eqref{21}. If $h \in C^{0, \alpha} (\bar D)$ is harmonic in
$D$, then a unique solution $u$ of this equation belongs to $C^{0, \alpha} (\bar D)$
and is $\mu$-panharmonic in $D$.
\end{theorem}

\begin{proof}
It is a classical result (see, for example, \cite{M}, Theorem~8.6.1) that an
$L^2$-solution of a weakly singular integral equation (it exists in our case) is in
$C^{0} (\bar D)$ provided the right-hand side term has this property. However, the
continuity of $u$ does not guarantee the existence of second derivatives of the
Newtonian potential~$T u$. Let us establish their existence under the assumptions
made.

Since $h \in C^{0, \alpha} (\bar D)$, the solution $u$ has the same property.
Indeed, writing the equation as follows
\begin{equation}
u = \mu^2 T u + h \, , \label{22}
\end{equation}
we see that both terms on the right are in $C^{0, \alpha} (\bar D)$, which is a
consequence of the following result (see \cite{V}, Lemma 2.3). If $u \in L^\infty
(D)$, then $T u \in C^{0, 1} (\bar D)$, that is, $T u$ is Lipschitz continuous. Now,
another classical result (see \cite{GT}, Lemma~4.2) yields that $T u \in C^{2} (D)$
and $\nabla^2 (T u) = u$. Furthermore, relation \eqref{22} implies that $u \in C^{2}
(D)$ since $h$ is harmonic in $D$. Then, applying the Laplacian to both sides of
\eqref{22}, we obtain that $u$ is $\mu$-panharmonic in $D$.
\end{proof}

In other words, for every $\mu^2 > 0$ there exists the bounded operator
\[ (I - \mu^2 T)^{-1} : L^2 (D) \to L^2 (D) \, ,
\]
which brings a $\mu$-panharmonic function belonging to $C^{0, \alpha} (\bar D)$ into
correspondence with each harmonic in $D$ function from the same H\"older space. It
is not clear whether the range of this operator comprises the whole set of
$\mu$-panharmonic functions. At the same time, Theorem~2.2 yields that every
nonnegative function belonging to this set has a pre-image in the cone of
nonnegative harmonic functions. Hence, all nonnegative $\mu$-panharmonic functions
are in the operator's range.

Now, we prove the converse of equality \eqref{MMtil}, which was announced
in \cite{Ku4}.

\begin{theorem}
Let $D \subset {\bf R}^m$, $m \geq 2$, be a bounded domain and let $u \in C^0 (D)$
be real-valued. If identity \eqref{MMtil} holds for every $x \in D$ and all $r \in
(0, r (x))$, where $r (x) > 0$ is such that the ball $B_{r (x)} (x)$ is admissible,
then $u$ is $\mu$-panharmonic in~$D$.
\end{theorem}

\begin{proof}
Let $\rho > 0$ be sufficiently small. If $r \in (0, \rho)$, then $M^\bullet (x,r,u)$
is defined for every $x$, which belongs to an open subset of $D$ depending on the
smallness of $\rho$. Moreover, $M^\bullet (x,r,u)$ is differentiable with respect to
$r$ and
\[ \partial_r M^\bullet (x,r,u) = m r^{-1} [ M^\circ (x,r,u) - M^\bullet (x,r,u) ] 
\quad \mbox{for} \ r \in (0, \rho) .
\]
In view of \eqref{18}, this takes the form:
\[ \frac{\partial_r M^\bullet}{M^\bullet} = \mu \frac{I_{(m-2)/2} (\mu r)}
{I_{m/2} (\mu r)} - \frac{m}{r} = \mu \frac{I_{m/2}' (\mu r)} {I_{m/2} (\mu r)} -
\frac{m}{2 r} \, ,
\]
where the last equality is a consequence of the recurrence formula (\cite{Wa},
p.~79):
\[ I_{\nu-1} (z) = I_{\nu}' (z) + \frac{\nu}{z} I_{\nu} (z) \, .
\]

Since the equation for $M^\bullet$ has logarithmic derivatives on both sides,
integrating with respect to $r$ over the interval $(\epsilon, \rho)$, we obtain,
after letting $\epsilon \to 0$, relation \eqref{MM'} with $r$ changed to $\rho$.
Indeed, shrinking $B_\epsilon (x)$ to its centre on the left-hand side, we see that
$M^\bullet (x,\epsilon,u) \to u (x)$ because $u \in C^0 (D)$, and this takes place
for every $x$ in an arbitrary closed subset of $D$. By letting $\epsilon \to 0$ on
the right-hand side, the factor $\Gamma \left( \frac{m}{2} + 1 \right)$ arises due
to the leading term of the power expansion of $I_{m/2}$. Hence we obtain
\[ M^\bullet (x,\rho,u) = a^\bullet (\mu \rho) \, u (x)
\]
for every $x \in D$ and all admissible $\rho$.

After changing $\rho$ to $r$ in this equality, we differentiate it with respect to
$r$, which yields, after some rearrangement, that identity \eqref{MM} holds for
every $x \in D$ and all $r \in (0, r (x))$, where $r (x) > 0$ is such that the ball
$B_{r (x)} (x)$ is admissible. Then $u$ is $\mu$-panharmonic in $D$ by Theorem~4.2
proved in \cite{Ku1}.
\end{proof}

\vspace{-14mm}

\renewcommand{\refname}{
\begin{center}{\Large\bf References}
\end{center}}
\makeatletter
\renewcommand{\@biblabel}[1]{#1.\hfill}
\makeatother

\end{document}